\documentclass{amsart}

\usepackage{ etoolbox, mathtools }

\usepackage[defaultlines=4,all]{nowidow}

\newcommand{\la}{\lambda}
    
\usepackage[section]{placeins}

\usepackage[usenames,dvipsnames,svgnames,table]{xcolor}
\definecolor{myblue}{HTML}{5555f0}
\definecolor{paleblue}{HTML}{bbbbf9}
\definecolor{myred}{HTML}{f03f02}

\newcommand{\grey}{darkgray}
\newcommand{\blue}{myblue}
\newcommand{\paleblue}{paleblue}
\newcommand{\red}{myred}

\usepackage{ytableau}

\usepackage{tikz}
\usetikzlibrary{decorations.pathreplacing, decorations.markings, shapes.geometric, arrows.meta, shapes.misc, patterns, calc}

\tikzset{%
every picture/.append style={x={(0,-0.85cm)},y={(0.85cm,0)}},
mid arrow/.style={postaction={decorate,decoration={
    markings,
    mark=at position 0.6 with {\arrow{Stealth}}
}}},
half mid arrow/.style={postaction={decorate,decoration={
    markings,
    mark=at position 0.58 with {\arrow{Stealth[harpoon]}}
}}},
}

\newcommand{\drawnode}[1]{%
    \fill[\grey] #1 circle (1pt);
}
\newcommand{\nudge}{0.015}  

\usepackage[shortlabels]{ enumitem }

\usepackage[margin=0.25cm]{caption}
\usepackage[labelformat=simple]{subcaption}

\usepackage{ amsthm, thmtools }

\definecolor{dark-blue}{rgb}{0.15,0.15,0.4}
\usepackage[colorlinks=true, linkcolor=black, citecolor=black, urlcolor=dark-blue, anchorcolor=dark-blue]{ hyperref }

\declaretheorem[style=plain,numberwithin=section]{theorem}
\declaretheorem[style=plain,numberlike=theorem]{lemma}
\declaretheorem[style=plain,numberlike=theorem]{proposition}

\declaretheorem[style=remark,numberlike=theorem]{remark}
\declaretheorem[style=definition,numberlike=theorem]{definition}
\declaretheorem[style=definition,numberlike=theorem]{example}
\declaretheorem[name=Example,style=definition,numbered=no]{example*}

\makeatletter
\renewcommand{\sectionautorefname}{\S\@gobble}
\renewcommand{\subsectionautorefname}{\S\@gobble}
\makeatother

\newcommand{\initialsspace}{0.1em}
 
\makeatletter
\newcommand{\splitlist}[1]{\@splitlist#1\@nil}
\def\@splitlist#1\@nil{%
  \if\relax\detokenize{#1}\relax
    \expandafter\@gobble
  \else
    \expandafter\@firstofone
  \fi
  {\@spl@tlist#1.\@nil}%
}
 
\def\@spl@tlist#1.#2\@nil{%
    \def\tmpA{#1}%
    \def\tmpB{#2}%
    \def\tmpP{.}%
    \ifx\tmpB\tmpP%
        #1.%
    \else{%
        \ifx\tmpA\@empty%  
        \else%
                #1.\nobreak\hspace{\initialsspace}%
        \fi%
    }% 
    \fi%
  \if\relax\detokenize{#2}\relax
    \expandafter\@firstoftwo
  \else
    \expandafter\@secondoftwo
  \fi
  {\unskip}%
  {\@spl@tlist#2\@nil}%
}
\makeatother

\usepackage{ amssymb, mleftright, siunitx }
\usepackage{ blkarray }

\NewDocumentCommand\set{s m}{%
    \IfBooleanTF#1%
    {\left\{ #2 \right\}}%
    {\{#2\}}%
}
\NewDocumentCommand\setbuild{s m m}{%
    \IfBooleanTF#1%
    {\ensuremath{\left\{\, #2 \, \middle| \, #3 \,\right\}}}%
    {\ensuremath{\{\, #2 \, \mid \, #3 \,\}}}%
}

\DeclarePairedDelimiter{\abs}{\lvert}{\rvert}
\makeatletter
\let\oldabs\abs
\def\abs{\@ifstar{\oldabs}{\oldabs*}}
\makeatother

\newcommand{\Z}{\mathbb{Z}}
\newcommand{\intervalwz}[1]{{[#1]}_0}

\newcommand{\Y}[1]{\operatorname{Y}(#1)}
\newcommand{\Ybox}[2]{\beth_{(#1,#2)}}

\newcommand{\qbinom}[2]{\binom{#1}{#2}_q}

\renewcommand{\leq}{\leqslant}
\renewcommand{\geq}{\geqslant}

\newcommand{\comp}{\mathsf{c}}

\newcommand{\Llattice}{\Lambda_{\text{L}}(\beta/\alpha)}
\newcommand{\Rlattice}{\Lambda_{\text{R}}(\beta/\alpha)}
\newcommand{\alphapartition}{(2,1,1,0,0,0)}
\newcommand{\betapartition}{(6,6,5,4,4,3)}

\title
[Flagged Schur polynomial duality]
{Flagged Schur polynomial duality \\ via a lattice path bijection}
\author{Eoghan McDowell}

\usepackage[foot]{amsaddr}

\makeatletter
\def\@setfoot@addresses{%
    \ifx\@affiliation\@empty\else
    \affiliationname\ \@affiliation \@addpunct.%
    \fi
    \ifx\web@page\@empty\else

    \webpagename\ \web@page \@addpunct.%
    \fi
}
\def\affiliationname{\textit{Affiliation}:}
\def\affiliation#1{\gdef\@affiliation{#1}}
\let\@affiliation\@empty
\def\webpagename{\textit{Webpage}:}
\def\webpage#1{\gdef\web@page{#1}}
\let\web@page\@empty
\makeatother

\email{eoghan.mcdowell@oist.jp} 
\affiliation{Okinawa Institute of Science and Technology}
\webpage{\href{https://eoghanjmcdowell.com}{eoghanjmcdowell.com}}

\makeatletter
\def\@setsubjclass{%
    \vspace{4pt}

    \vspace{-\baselineskip}
    {\itshape\subjclassname.}\enspace\@subjclass\@addpunct.%
}
\makeatother


\makeatletter
\def\@setthanks{%
\vspace{-\baselineskip}\vspace{4pt}
\def\thanks##1{\@par##1\@addpunct.}
\thankses%
}
\def\journalinfo#1{\thanks{#1}}
\makeatother

\journalinfo{
This is the accepted manuscript for an article published in Electronic Journal of Combinatorics 30(1) (2023) P1.5, available online at \href{https://doi.org/10.37236/11200}{doi.org/10.37236/11200}%
}

\begin{document}

\begin{abstract}
This paper proves an identity between flagged Schur polynomials, giving a duality between row flags and column flags.
This identity generalises both the binomial determinant duality theorem due to Gessel and Viennot and the symmetric function duality theorem due to Aitken.
As corollaries we obtain the lifts of the binomial determinant duality theorem to \(q\)-binomial coefficients and to symmetric polynomials.
Our method is a path counting argument on a novel lattice generalising that used by Gessel and Viennot.
\end{abstract}

\maketitle

\section{Introduction}
\label{section:intro}

We generalise the following identities, due respectively to Gessel and Viennot and to Aitken.
For \(n\) a nonnegative integer, we write \(\intervalwz{n}=\set{0,1,\ldots,n}\).

\begin{restatable}[Binomial determinant duality theorem {\cite[Proposition 7]{gessel1985paths}}]{theorem}{binomialidentity}
\label{thm:binomial_identity}
Let \(n\) be a nonnegative integer, let \(A\) and \(B\) be subsets of \(\intervalwz{n}\) of equal size, and let \(A^\comp\) and \(B^\comp\) be their complements in \(\intervalwz{n}\). Then
\[
    \det\mleft(%
        \binom{b}{a}%
    \mright)_{a \in A, b \in B}
    =
    \det\mleft(%
        \binom{a'}{b'}%
    \mright)_{a' \in A^\comp, b' \in B^\comp}
    .
\]
\end{restatable}

\begin{restatable}[Symmetric function duality theorem {\cite[\S2]{aitken1931duality}}]
{theorem}{symfunctionidentity}
\label{thm:sym_function_identity}
Let \(n\) be a nonnegative integer, let \(A\) and \(B\) be subsets of \(\intervalwz{n}\) of equal size, and let \(A^\comp\) and \(B^\comp\) be their complements in \(\intervalwz{n}\).
Then
\[
    \det\Bigl(%
        h_{b - a} %
    \Bigr)_{a \in A, b \in B}
    =
    \det\Bigl(%
        e_{a' - b'} %
    \Bigr)_{a' \in A^\comp, b' \in B^\comp}
    .
\]
\end{restatable}

Here \(h_i\) and \(e_i\) denote the complete homogeneous and elementary symmetric functions of degree \(i\) respectively, where by convention \(e_0 = h_0 = 1\) and \(e_d = h_d = 0\) for \(d < 0\).
For the purposes of indexing matrices, we consider finite subsets of \(\Z\) to be ordered smallest element to largest.

We state our main theorem here in a slightly weakened form to avoid a technical condition whose necessity becomes clear only in the proof; the full statement with a weaker requirement on the parameters is given in \autoref{thm:main_theorem}.

\begin{restatable}{theorem}{maintheorem}\label{thm:intro_det_identity}
Let \(n\) be a nonnegative integer, let \(A\) and \(B\) be subsets of \(\intervalwz{n}\) of equal size, and let \(A^\comp\) and \(B^\comp\) be their complements in \(\intervalwz{n}\).
Let \(\alpha\) and \(\beta\) be partitions with \(n\) parts (with parts equal to \(0\) permitted) such that \(\alpha_i \leq \beta_i\) for all \(i \in [n]\).
Suppose each part of \(\alpha\) and \(\beta\) is at most \(1\) less than the preceding part.
Then the determinants
\[
    \det\Bigl(%
        h_{b - a} (x_{\alpha_{a+1}+1}, x_{\alpha_{a+1}+2}, \ldots, x_{\beta_{b}}) %
    \Bigr)_{a \in A, b \in B}
\]
and
\[
    \det\Bigl(%
        e_{a' - b'} (x_{\alpha_{a'}+1}, x_{\alpha_{a'}+2}, \ldots, x_{\beta_{b'+1}}) %
    \Bigr)_{a' \in A^\comp, b' \in B^\comp}
\]
are equal.
\end{restatable}
We illustrate this statement in \autoref{eg:intro_thm_example} at the end of this section.
Note that symmetric functions of positive degree over an empty set of variables are considered to equal~\(0\).

The determinants in \autoref{thm:intro_det_identity} are equal to \emph{flagged (skew) Schur polynomials}.
These polynomials were introduced by Lascoux and Sch\"{u}tzenberger \cite{lascoux1982Schubert}; we recall the definitions in \autoref{subsection:flagged_polys}.
Our first determinant equals a row-flagged polynomial; the second, a column-flagged polynomial.
We therefore obtain the following duality theorem for flagged skew Schur polynomials (see  \autoref{subsection:flagged_polys} for the definitions and a version with weaker conditions on the parameters).

\begin{theorem}\label{thm:intro_flagged_polynomial_identity}
Let \(n\), \(A\) and \(B\) be as in the statement of \autoref{thm:intro_det_identity}.
% Let \(n\) be a nonnegative integer, let \(A\) and \(B\) be subsets of \(\intervalwz{n}\) of equal size \(l\), and let \(A^\comp\) and \(B^\comp\) be their complements in \(\intervalwz{n}\).
% Let \(\alpha\) and \(\beta\) be partitions with \(n\) parts (with parts equal to \(0\) permitted) such that \(\alpha_i \leq \beta_i\) for all \(i \in [n]\).
% Suppose each part of \(\alpha\) and \(\beta\) is at most \(1\) less than the preceding part.
Set:
\begin{align*}
    \mu_i   &= A_{l+1-i} + i - l, &\, 
    \la_i       &= B_{l+1-i} + i - l, \\
    f_i         &= \alpha_{A_{l+1-i}+1}+1, &\,
    g_i         &= \beta_{B_{l+1-i}}, \\
    f^\ast_i    &= \alpha_{A^\comp_{i}}+1, &\,
    g^\ast_i    &= \beta_{B^\comp_{i}+1}.
\end{align*}
Then
\[
    s_{\la/\mu}(f,g) = s^\ast_{\la'/\mu'}(f^\ast, g^\ast).
\]
\end{theorem}

% The first determinant equals a row-flagged function, with shape given by \(B\) and \(A\) (modified into partitions by adding staircases), while the second equals a column-flagged function, with shape given by \(A^\comp\) and \(B^\comp\) (modified into partitions by adding staircases), both with flags given by \(\beta\) and \(\alpha\).
% 

Our proof of our identity, which comprises \autoref{section:proof}, is by counting paths on lattices using the Lindstr\"{o}m--Gessel--Viennot lemma and constructing a bijection between sets of such paths.
This is the approach used by Gessel and Viennot to prove \autoref{thm:binomial_identity}.
However, we obtain an identity more general than theirs by applying the lemma to lattices of more general shape and with weighted edges.
Weighted edges allow us to replace binomial coefficients with symmetric polynomials;
the more general shape of our lattices allows us to vary the number of variables in each symmetric polynomial. 

We explain how \autoref{thm:binomial_identity} and \autoref{thm:sym_function_identity} can be deduced from \autoref{thm:intro_det_identity} in \autoref{section:corollaries}.
Gessel and Viennot's binomial duality theorem corresponds to a staircase-shaped lattice, while Aitken's symmetric function duality theorem corresponds to a rectangular lattice.
This paper thus provides a unifying framework for these two results, and since our theorem also permits lattices of shapes intermediate between staircases and rectangles, it generalises them significantly.

We also obtain lifts of the binomial duality theorem to \(q\)-binomial coefficients and to symmetric polynomials, which we state here and prove in \autoref{section:corollaries}.

\begin{restatable}
{corollary}{qbinomialidentity}
\label{cor:q-binomial_identity}
Let \(n\) be a nonnegative integer, let \(A\) and \(B\) be subsets of \(\intervalwz{n}\) of equal size, and let \(A^\comp\) and \(B^\comp\) be their complements in \(\intervalwz{n}\).
Then
\[
    \det\mleft(%
        \qbinom{b}{a}%
    \mright)_{a \in A, b \in B}
    =
    \det\mleft(%
        q^{\binom{a'-b'}{2}} \qbinom{a'}{b'}%
    \mright)_{a' \in A^\comp,b' \in B^\comp}
    .
\]
\end{restatable}

\begin{restatable}
{corollary}{sympolybinomidentity}
\label{cor:sym_poly_binomial_identity}
Let \(n\) be a nonnegative integer, let \(A\) and \(B\) be subsets of \(\intervalwz{n}\) of equal size, and let \(A^\comp\) and \(B^\comp\) be their complements in \(\intervalwz{n}\).
Then
\[
    \det\Bigl(%
        h_{b - a} (x_{1}, x_{2}, \ldots, x_{a+1}) %
    \Bigr)_{a \in A, b \in B}
    =
    \det\Bigl(%
        e_{a' - b'} (x_{1}, x_{2}, \ldots, x_{a'}) %
    \Bigr)_{a' \in A^\comp, b' \in B^\comp}
    .
\]
\end{restatable}

Aitken's proof of \autoref{thm:sym_function_identity} \cite{aitken1931duality} is of an entirely different flavour to the methods in this paper: it is an application of an elementary result in linear algebra due to Jacobi, which we refer to as Jacobi's complementary minor formula.
This formula also corresponds to duality between certain combinatorial models of orthogonal polynomials \cite[I\S5]{viennot1985orthogonal}.
However, we show in \autoref{section:JT_connection} that Jacobi's formula is insufficient to prove our main theorem.

\begin{example}
\label{eg:intro_thm_example}
Suppose \(n=4\), \(A = \set{0,1,2}\), \(B = \set{1,3,4}\), \(\alpha = (1,1,0,0)\) and \(\beta = (4,3,3,2)\).
Then the two determinants that \autoref{thm:intro_det_identity} states are equal are
\[
\hspace*{-6.5em}
\begin{blockarray}{cc>{\scriptstyle}c>{\scriptstyle}c>{\scriptstyle}c}
      &   & b=1 & b=3 & b=4 \\
      &   & \beta_{b}=4 & \beta_{b}=3 & \beta_{b}=2\\[4pt]
    \begin{block}{>{\scriptstyle}c>{\scriptstyle}c|ccc|}
    a=0 & \alpha_{a+1}+1=2 & h_1(x_2, x_3, x_4) & h_3(x_2, x_3) & h_4(x_2) \\
    a=1 & \alpha_{a+1}+1=2 & 1 & h_2(x_2, x_3) & h_3(x_2) \\
    a=2 & \alpha_{a+1}+1=1 & 0 & h_1(x_1, x_2, x_3) & h_2(x_1, x_2) \\
    \end{block}
\end{blockarray}
\]
and
\[
\hspace*{-6.5em}
\begin{blockarray}{ccc>{\scriptstyle}c>{\scriptstyle}cc}
      &   && b'=0 & b'=2 &\\
      &   && \beta_{b'+1}=4 & \beta_{b'+1}=3& \\[4pt]
    \begin{block}{>{\scriptstyle}c>{\scriptstyle}cc|cc|c}
    a'=3 & \alpha_{a'}+1=1 && e_3(x_1, x_2, x_3, x_4) & e_1(x_1, x_2, x_3)& \\
    a'=4 & \alpha_{a'}+1=1 && e_4(x_1, x_2, x_3, x_4) & e_2(x_1, x_2, x_3)&.\\
    \end{block}
\end{blockarray}
\]
Restating in terms of flagged Schur polynomials, the first determinant above is equal to \(s_{(2,2,1)}((1,2,2), (2,3,4))\) and the second is \(s^\ast_{(3,2)}((1,1), (4,3))\) (in the notation of \autoref{thm:intro_flagged_polynomial_identity}, we have \(\la = (2,2,1)\), \(\mu= (0^3)\), \(f = (1, 2, 2)\), \(g = (2, 3, 4)\), \(f^\ast = (1, 1)\) and \(g^\ast = (4, 3)\)).
\autoref{thm:intro_flagged_polynomial_identity} states that these flagged Schur polynomials are equal.
% \[
% s_{(2,2,1)}((1,2,2), (2,3,4)) = s^\ast_{(3,2)}((1,1), (4,3)).
% \]
Equivalently (since the necessary conditions of \cite[Theorems 3.5 and 3.5*]{wachs1985flaggedSchur} are met), the sets of flagged tableaux of the form
\[
% \ytableausetup{}
\begin{gathered}\begin{ytableau}
    \none    & \none    & \none       & \none \\
    \none[1] &          &             & \none[2]\\
    \none[2] &          &             & \none[3] \\
    \none[2] &          &    \none    & \none[4] \\
    \none    & \none    & \none       & \none
\end{ytableau}\end{gathered}
\text{\qquad and \quad} 
\begin{gathered}\begin{ytableau}
    \none    & \none[1] & \none[1]    & \none \\
    \none    &          &             & \none \\
    \none    &          &             & \none \\
    \none    &          &    \none    & \none \\
    \none    & \none[4] & \none[3]    & \none
\end{ytableau}\end{gathered}
\]
have equal generating functions.
\end{example}

\section{Proof of the main theorem}
\label{section:proof}

The steps in our proof are:
\begin{enumerate}[label={\arabic*)}]
    \item
    define two lattices and paths upon them;
    \item
    show that the weighted counts of tuples of paths on the lattices equal the determinants in the main theorem;
    \item
    construct a weight-preserving bijection from tuples of paths on one lattice to the other.
\end{enumerate}

We adopt the notation of \autoref{thm:intro_det_identity} throughout: let \(n\) be a nonnegative integer, let \(A\) and \(B\) be subsets of \(\intervalwz{n}\) of equal size with complements in \(\intervalwz{n}\) denoted \(A^\comp\) and \(B^\comp\), and let \(\alpha\) and \(\beta\) be partitions with \(n\) parts (with parts equal to zero permitted) such that \(\alpha_i \leq \beta_i\) for all \(i \in [n]\)
(though we do not assume that each part is at most \(1\) less than the preceding part).
Additionally, let \(l = \abs{A} = \abs{B}\), and let \(r = \abs{A^\comp} = \abs{B^\comp} = n+1-l\).

\subsection{Definition of lattices}
\label{section:lattices}

We picture Young diagrams as lying in a plane with the \(x\)-direction being downward and the \(y\)-direction being rightward, and the \(1 \times 1\) square whose bottom-right corner is the point \((i,j)\) is referred to as the \emph{box} \(\Ybox{i}{j}\).
(Though it is common to refer to a box simply by its coordinates, we denote boxes in this manner since we have need to distinguish between boxes and points.)
The skew Young diagram of \(\beta/\alpha\) is then \(\Y{\beta/\alpha} = \setbuild{ \Ybox{i}{j} }{1 \leq i \leq n,\ \alpha_i +1 \leq j \leq \beta_i}\).

We use \(\Y{\beta/\alpha}\) to construct two lattices, \(\Llattice\) and \(\Rlattice\).
In both lattices, the set of nodes is the set of points in the plane which occur as some corner of some box in \(\Y{\beta/\alpha}\).
The edges in each lattice are described below.
We give some edges a \emph{weight} which will be a formal variable.
The weight of a path is then given by the product of the weights of the steps, and the weight of a tuple of paths is the product of the weights of each path.

In \(\Llattice\), we have horizontal edges and vertical edges:
\begin{itemize}
    \item the horizontal edges are the horizontal sides of the boxes in \(\Y{\beta/\alpha}\), and are directed rightward;
    \item the vertical edges are the right-hand sides of the boxes in \(\Y{\beta/\alpha}\), are directed downward, and have weight \(x_j\) where \(j\) is the column of the corresponding box.
\end{itemize}
In  \(\Rlattice\), we have horizontal edges and diagonal edges:
\begin{itemize}
    \item
    the horizontal edges are the horizontal sides of the boxes in \(\Y{\beta/\alpha}\), and are directed leftward;
    \item
    the diagonal edges are the top-right to bottom-left diagonals of the boxes in \(\Y{\beta/\alpha}\), are directed left-and-downward, and have weight \(x_j\) where \(j\) is the column of the corresponding box.
\end{itemize}
An example of each of these lattices is depicted in \autoref{fig:lattices}.

\begin{figure}[htb]
    \centering
    \begin{subfigure}{0.49\linewidth}
        \centering
        \includegraphics[scale=0.8]{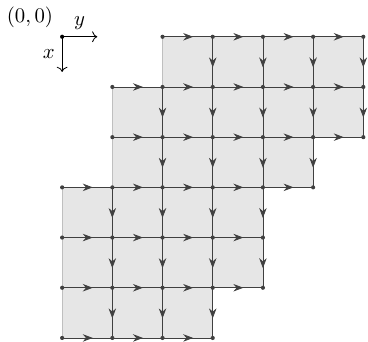}
        \caption{\(\Llattice\)}
        \label{subfig:L-lattice}
    \end{subfigure}
    \begin{subfigure}{0.49\linewidth}
        \centering
        \includegraphics[scale=0.8]{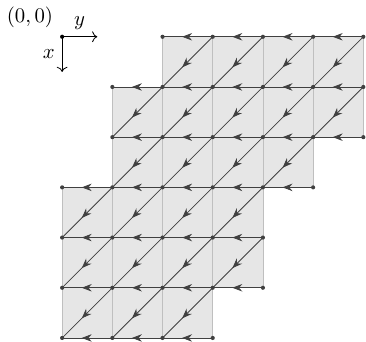}
        \caption{\(\Rlattice\)}
        \label{subfig:R-lattice}
    \end{subfigure}
    \caption{Our lattices when \(n=6\), \(\alpha = \alphapartition\) and \(\beta = \betapartition\).
    }
    \label{fig:lattices}
\end{figure}

We now define \emph{sources} and \emph{sinks} for paths on these lattices.

In \(\Llattice\), take as \emph{sources} the left-most node in each horizontal line indexed by \(A\),
and take as \emph{sinks} the right-most node in each horizontal line indexed by \(B\).
Explicitly, the sources are the points \(\setbuild{(a, \alpha_{a+1})}{a \in A}\) (where we interpret \(\alpha_{n+1} = \alpha_n\) if \(n \in A\)) and the sinks are the points \(\setbuild{(b, \beta_{b})}{b \in B}\) (where we interpret \(\beta_0 = \beta_1\) if \(0 \in B\)).

In \(\Rlattice\), take as \emph{sources} the right-most node in each horizontal line indexed by \(B^\comp\),
%(that is, the bottom-right corner of the right-most box in each row indexed by \(B^\comp\))
and take as \emph{sinks} the left-most node in each horizontal line indexed by \(A^\comp\). %(that is, the top-right corner of the right-most box in each row indexed by \(A^\comp+1\)).
Explicitly, the sources are the points \(\setbuild{(b', \beta_{b'})}{b' \in B^\comp}\) (where we interpret \(\beta_0 = \beta_1\) if \(0 \in B^\comp\)) and the sinks are the points \(\setbuild{(a', \alpha_{a'+1})}{a' \in A^\comp}\) (where we interpret \(\alpha_{n+1} = \alpha_n\) if \(n \in A^\comp\)).

We are interested in tuples of paths in \(\Llattice\) and \(\Rlattice\) that join the sources to the sinks in a matching; we call such tuples in \(\Llattice\) \emph{blue connectors} and such tuples in \(\Rlattice\) \emph{red connectors}.
We furthermore describe sources, sinks and paths in \(\Llattice\) as \emph{blue} and in \(\Rlattice\) as \emph{red}.
A blue connector and a red connector are depicted in \autoref{fig:paths}.

\begin{figure}[htb]
    \centering
    \begin{subfigure}{0.495\linewidth}
        \centering
        \includegraphics[scale=0.8]{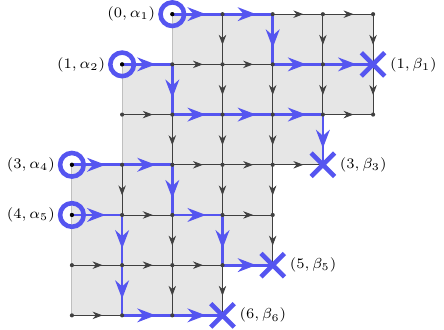}
        \caption{A blue connector.}
        \label{subfig:l-path}
    \end{subfigure}
    \begin{subfigure}{0.495\linewidth}
        \centering
        \includegraphics[scale=0.8]{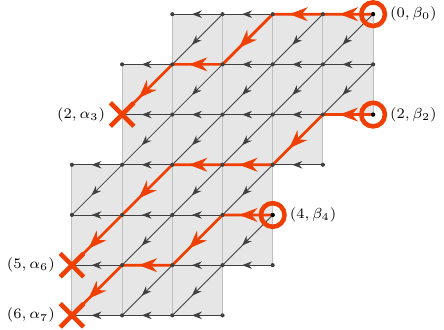}
        \caption{A red connector.}
        \label{subfig:r-path}
    \end{subfigure}
    \caption{Examples of a blue connector and a red connector, when \(n=6\), \(\alpha = \alphapartition\), \(\beta = \betapartition\),
    \(A = \set{0,1,3,4}\) and \(B = \set{1,3,5,6}\).
    Sources are indicated by circles, sinks by crosses.
    From the top of the diagrams to the bottom, the blue paths have weights \(x_4\), \(x_2x_5\), \(x_2x_3\) and \(x_1^2\) respectively, and the red paths have weights \(x_2x_4\), \(x_1x_2x_5\) and \(x_1x_3\) respectively.
    The weight of this blue connector, being the product of the weights of the blue paths, is \(x_1^2x_2^2x_3x_4x_5\), which is also the weight of this red connector.
    }
    \label{fig:paths}
\end{figure}

\subsection{Enumeration of connectors}

We count the \emph{non-intersecting} blue connectors and red connectors.
They key result we use for this is the Lindstr\"{o}m--Gessel--Viennot Lemma, stated below.
The lemma was first articulated in the context of Markov chains in \cite{karlin1959coincidence}, and later in the context of matroid theory in \cite{lindstroem1973matroids}.
It was subsequently used to deduce various combinatorial identities in \cite{gessel1985paths} and has had applications to the theory of orthogonal polynomials \cite{viennot1985orthogonal}.
A weighted combinatorial statement is given in \cite{stembridge1990nonintersecting}; for an illuminating illustration of the argument behind the lemma, see \cite{benjamin2005countingOnDeterminants}.

\begin{theorem}[Lindstr\"om--Gessel--Viennot Lemma]
\label{thm:lgv_lemma}
Let \(G\) be a directed acyclic graph with \(m\) designated sources and sinks, where \(m\) is a nonnegative integer.
Let \(M\) be the \(m \times m\) matrix whose \((i, j)\)th entry is the number of paths, counted with weight, from the \(i\)th source to the \(j\)th sink.
Suppose \(G\) is nonpermutable.
Then the number of non-intersecting \(m\)-tuples of paths from sources to sinks, counted with weight, is equal to the determinant of \(M\).
\end{theorem}

Here, \emph{nonpermutable} means that a non-intersecting \(m\)-tuple of paths must connect the \(i\)th source to the \(i\)th sink.
The lattices \(\Llattice\) and \(\Rlattice\) are clearly acyclic and nonpermutable, so \autoref{thm:lgv_lemma} applies.
We thus want to count paths from the \(i\)th source to the \(j\)th sink in each lattice.

\begin{proposition}
\label{prop:l-path_count}
The weighted count of non-intersecting blue connectors in the lattice \(\Llattice\) is
\[
    \det\mleft(%
        h_{b - a}(x_{\alpha_{a+1}+1}, x_{\alpha_{a+1}+2}, \ldots, x_{\beta_{b}}) %
    \mright)_{a \in A,b \in B}.
\]
\end{proposition}

\newcommand{\rightbox}{(b, \beta_{b})}
\newcommand{\leftbox}{(a, \alpha_{a+1})}

\begin{proof}
By \autoref{thm:lgv_lemma}, it suffices to show that the weighted count of paths from \(\leftbox\) to \(\rightbox\) is \(h_{b - a}(x_{\alpha_{a+1}+1}, x_{\alpha_{a+1}+2}, \ldots, x_{\beta_{b}})\), for all \(a \in A\) and \(b \in B\).

Suppose first that \(a \geq b\).
If the inequality is strict, then \(\leftbox\) is below \(\rightbox\), so there are no paths between them, and \(h_{b - a} = 0\) as required.
If equality holds, then \(\leftbox\) and \(\rightbox\) are in the same row so there is a unique horizontal path between them, and \(h_{b-a} = 1\) as required.

Suppose next \(a < b\) and \(\alpha_{a+1} \geq \beta_{b}\).
Then \(\leftbox\) is further right that \(\rightbox\) (or in the same column, but without vertical edges joining them), so there are no paths between them.
Meanwhile the corresponding polynomial is a symmetric function over an empty range of variables and hence is zero, as required.

Now suppose \(a < b\) and \(\alpha_{a+1} < \beta_{b}\).
A path from \(\leftbox\) to \(\rightbox\) must make exactly \(b - a\) vertical steps.
Observe that the vertical steps in such a path must be made down (the right-hand sides of) boxes in the rectangular region whose vertices are the boxes \(\Ybox{a+1}{\alpha_{a+1}+1}\), \(\Ybox{a+1}{\beta_{b}}\), \(\Ybox{b}{\beta_{b}}\) and \(\Ybox{b}{\alpha_{a+1}+1}\), as exemplified in \autoref{fig:rectangular_region}.
Since \(\alpha\) and \(\beta\) are partitions, we have \(\alpha_{i} \leq \alpha_{a+1}\) and \(\beta_{i} \geq \beta_{b}\) for all \(a+1 \leq i \leq b\), so all these boxes are indeed contained in the Young diagram \(\Y{\beta/\alpha}\).

\begin{figure}[htb]
    \centering
    \includegraphics{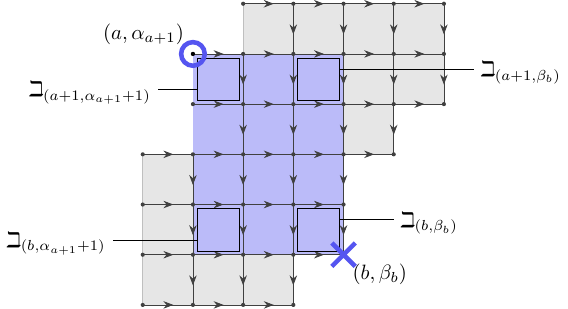}
    \caption{The collection of boxes which must contain the vertical steps of any path from \((a, \alpha_{a+1})\) to \((b,\beta_b)\) when \(a=1\), \(b=5\), \(n=6\), \(\alpha = \alphapartition\) and \(\beta = \betapartition\).
    }
    \label{fig:rectangular_region}
\end{figure}

Thus the choice of \(b-a\) columns in which vertical steps take place can be made freely with repetition from \(\alpha_{a+1}+1, \alpha_{a+1}+2, \ldots, \beta_{b}\) (and each such choice uniquely determines a path).
Then since the vertical edges in column \(j\) each have weight \(x_j\), the weighted count of possible paths is the required polynomial.
\end{proof}

\newcommand{\parallelogramcondition}{%
Suppose that for all \(a' \in A^\comp\) and \(b' \in B^\comp\), either:
\begin{itemize}
    \item
    \(a' - b' \leq 0\); or
    \item
    \(a' - b' > \beta_{b'+1} - \alpha_{a'}\); or
    \item
    for all \(i\) such that \(b'+1 \leq i \leq a'\), we have \(\alpha_{i} - \alpha_{a'} \leq a'-i\) and \(\beta_{b'+1} - \beta_i \leq i-b'-1\).
\end{itemize}
}

\begin{samepage}
\begin{proposition}
\label{prop:r-path_count}
\parallelogramcondition
Then the weighted count of non-intersecting red connectors in the lattice \(\Rlattice\) is
\[
    \det\mleft(%
        e_{a' - b'} (x_{\alpha_{a'}+1}, x_{\alpha_{a'}+2}, \ldots, x_{\beta_{b'+1}}) %
    \mright)_{a' \in A^\comp,b' \in B^\comp}.
\]
\end{proposition}
\end{samepage}

\renewcommand{\rightbox}{(b', \beta_{b'})}
\renewcommand{\leftbox}{(a', \alpha_{a'+1})}

\begin{proof}
By \autoref{thm:lgv_lemma}, it suffices to show that the weighted count of paths from \(\rightbox\) to \(\leftbox\) is \(e_{a' - b'} (x_{\alpha_{a'}+1}, x_{\alpha_{a'}+2}, \ldots, x_{\beta_{b'+1}})\), for all \(a' \in A^\comp\) and all \(b' \in B^\comp\).

Suppose first that \(a' \leq b'\).
If the inequality is strict, then \(\rightbox\) is below \(\leftbox\), so there are no paths between them, and \(e_{a' - b'} = 0\) as required.
If equality holds, then \(\rightbox\) and \(\leftbox\) are in the same row so there is a unique horizontal path between them, and \(e_{a'-b'} = 1\) as required.

Suppose next that
\(a' - b' > \beta_{b'+1} - \alpha_{a'}\).
Then any path starting at \(\rightbox\) reaches the left-hand side of the lattice and terminates before it reaches the \(a'\)th row, so there are no paths between \(\rightbox\) and \(\leftbox\). Meanwhile the corresponding polynomial is an elementary symmetric function in fewer variables than its degree and hence is zero, as required.

Now suppose
\(a' - b' \leq \beta_{b'} - \alpha_{a'}\).
A path from \(\rightbox\) to \(\leftbox\) must make exactly \(a'- b'\) diagonal steps.
Observe that the diagonal steps in such a path must be made across boxes in the parallelogram-shaped region whose vertices are the boxes \(\Ybox{b'+1}{\beta_{b'+1}}\), \(\Ybox{b'+1}{\alpha_{a'}+a'-b'}\), \(\Ybox{a'}{\alpha_{a'}+1}\) and \(\Ybox{a'}{\beta_{b'+1}+b'-a'+1}\), as exemplified in \autoref{fig:parallelogram_region}.
The condition for this collection of boxes to be contained in the Young diagram is that for all \(i \in \set{b'+1, b'+2, \ldots, a'}\) we have \(\alpha_i \leq \alpha_{a'} + (a'-i)\) and \(\beta_i \geq \beta_{b'+1} - (i-b'-1)\), which is exactly the assumed condition on \(\alpha\) and \(\beta\).

\begin{figure}[htb]
    \centering
    \includegraphics{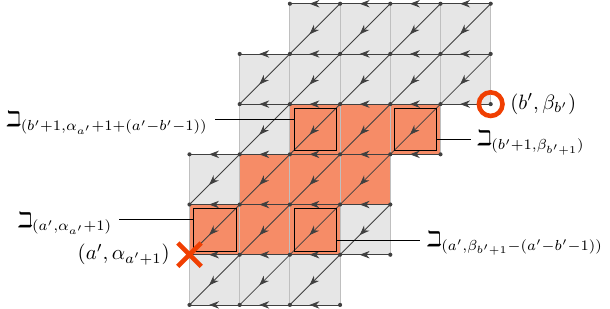}
    \caption{The collection of boxes which must contain the diagonal steps of any path from \((b', \beta_{b'})\) to \((a',\alpha_{a'+1})\) when \(a'=5\), \(b'=2\), \(n=6\), \(\alpha = \alphapartition\) and \(\beta = \betapartition\).
    }
    \label{fig:parallelogram_region}
\end{figure}

Thus the choice of \(a'-b'\) columns in which diagonal steps take place can be made freely without repetition from \(\alpha_{a'}+1\), \(\alpha_{a'}+2\), \ldots, \(\beta_{b'+1}\) (and each such choice uniquely determines a path).
Then since the diagonal edges in column \(j\) each have weight \(x_j\), the weighted count of possible paths is the required polynomial.
\end{proof}

\subsection{Bijection between connectors}
\label{section:path_bijection}

We now define a bijection between the non-intersecting blue connectors in \(\Llattice\) and the non-intersecting red connectors in \(\Rlattice\), for any skew-partition \(\beta/\alpha\).
(In fact, the arguments in this section hold for \(\alpha\) and \(\beta\) any compositions such that \(\alpha_i \leq \beta_i\) for all \(i\).)

It is convenient to overlay the lattices \(\Llattice\) and \(\Rlattice\), so we can compare paths on one with paths on the other.

The bijection is via the following construction.
We define the construction of a red connector from a blue connector; the inverse construction is analogous.

\begin{definition}
Given a non-intersecting blue connector, define the \emph{complementary} red connector to be the collection of paths constructed as follows:
beginning at each red source, take a horizontal step from each node unless the blue connector takes a vertical step from that node,
in which case take a diagonal step.
\end{definition}

\begin{example}
The red connector in \autoref{subfig:r-path} is complementary to the blue connector in \autoref{subfig:l-path}, as illustrated in \autoref{fig:complementary_paths}.
\begin{figure}[htb]
    \centering
    \includegraphics{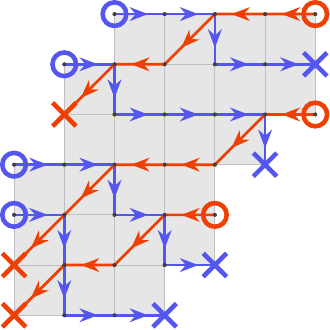}
    \caption{The blue connector and red connector from \autoref{fig:paths} overlayed, illustrating that the red connector is complementary to the blue connector.
    % The boxes in which vertical steps are made are highlighted in green.
    }
    \label{fig:complementary_paths}
\end{figure}
\end{example}

It is not obvious that the complementary red connector is indeed a red connector (that is, that each path reaches a distinct red sink).
We will show that it is, and that it is non-intersecting.
To do so, we use the following lemma.

\begin{lemma}
\label{lemma:complementary_lemma}
A non-intersecting blue connector and its complementary red connector intersect only at nodes from which a vertical blue step is taken.
\end{lemma}

\begin{proof}
Suppose, towards a contradiction, there exists a node which lies in both a non-intersecting blue connector and its complementary red connector and from which there is not a vertical blue step.
Consider a right-most such node \((i,j)\).

First observe \((i,j)\) is not a blue sink: in \(\Rlattice\), the right-most nodes are not the endpoints of any edges, so no path in \(\Rlattice\) starting at a red source can contain a blue sink.

Therefore there must be a horizontal blue step from \((i,j)\) to \((i,j+1)\).
In particular, \((i,j)\) is not right-most in its row, so it is not a red source.
Thus there is a red step to \((i,j)\), either horizontally from \((i,j+1)\) or diagonally from \((i-1,j+1)\).
We explain, and illustrate beneath each explanation, how both of these possibilities lead to a contradiction.

If the red step is horizontal, then by the construction of the complementary red connector there is no vertical blue step from \((i,j+1)\).
This contradicts our choice of \((i,j)\) as the right-most intersection from which there is not a vertical blue step.

\begin{center}
\begin{tikzpicture}
\draw[thick,color=\blue,postaction={half mid arrow}] (0.5-\nudge,0) -- (0.5-\nudge,1);
\draw[thick,color=\red,postaction={half mid arrow}] (0.5+\nudge,1) -- (0.5+\nudge,0);

\node at (0.5,2.4) {\(\implies\)};

\draw[thick,color=\blue,postaction={half mid arrow}] (0-\nudge,4) -- (0-\nudge,5);
\draw[thick,color=\red,postaction={half mid arrow}] (0+\nudge,5) -- (0+\nudge,4);
\draw[thick,color=\paleblue,postaction={mid arrow}] (0,5) -- (1,5);
\node[cross out, line width=0.9pt, draw=\grey, minimum size=0.32cm, at={(0.45,5)}] {};

\foreach \x/\y in {0.5/0, 0.5/1, 1/4, 1/5, 0/4, 0/5} {
    \drawnode{(\x,\y)}
}
\node[at={(0.5,0)},label={[shift={(0.8,0)}]\(\scriptstyle(i,j)\)}] {};
\node[at={(0,4)},label={[shift={(0.8,0)}]\(\scriptstyle(i,j)\)}] {};

\end{tikzpicture}
\end{center}

If the red step is diagonal, then by the construction of the complementary red connector, there is a vertical blue step from \((i-1, j+1)\) to \((i,j+1)\). This contradicts that the blue connector is non-intersecting.

\begin{center}
\begin{tikzpicture}
\draw[thick,color=\blue,postaction={mid arrow}] (0,0) -- (0,1);
\draw[thick,color=\red,postaction={mid arrow}] (-1,1) -- (0,0);

\node at (-0.5,2.4) {\(\implies\)};

\draw[thick,color=\blue,postaction={mid arrow}] (0,4) -- (0,5);
\draw[thick,color=\red,postaction={mid arrow}] (-1,5) -- (0,4);
\draw[thick,color=\blue,postaction={mid arrow}] (-1,5) -- (0,5);

\foreach \x/\y in {-1/0, -1/1, 0/0, 0/1, -1/4, -1/5, 0/4, 0/5} {
    \drawnode{(\x,\y)}
}
\node[at={(0,0)},label={[shift={(0.8,0)}]\(\scriptstyle(i,j)\)}] {};
\node[at={(0,4)},label={[shift={(0.8,0)}]\(\scriptstyle(i,j)\)}] {};

\end{tikzpicture}
\end{center}

In either case we have a contradiction, so no such intersection exists.
\end{proof}

\begin{proposition}
\label{prop:complements_exist}
The complementary red connector to a non-intersecting blue connector is a non-intersecting red connector.
\end{proposition}

\begin{proof}
First observe that in \(\Llattice\) the only edges directed out of the left-most nodes are horizontal, and thus the first step in every blue path is horizontal.
Therefore, by \autoref{lemma:complementary_lemma}, the complementary red connector does not contain any blue sources.
Since the blue sources and the red sinks partition the left-most nodes in the lattices, we deduce that the paths of the complementary red connector reach red sinks (not necessarily distinct).

We next show that the complementary red connector is non-intersecting.
This implies that the red sinks the red paths reach are distinct, and hence that it is indeed a red connector.

Suppose, towards a contradiction, that the complementary red connector has an intersection.
Consider a right-most intersection \((i,j)\).
Note that \((i,j)\) cannot be a red source (or a blue sink): if it were, then it would be the right-most node in its row, and so there would be no edges from \((i,j+1)\) or \((i-1,j+1)\) for a red path to arrive from (and so it could not be an intersection of two red paths).

Since we assumed \((i,j)\) to be a right-most intersection, the incoming red paths must come from distinct vertices: there is both a horizontal and a diagonal red step to \((i,j)\).
Then, by the construction of the complementary red connector and as illustrated below, there must be a vertical blue step to \((i,j+1)\) and there cannot be a vertical blue step out of \((i,j+1)\).
\begin{center}
\begin{tikzpicture}
\draw[thick,color=\red,postaction={mid arrow}] (0,1) -- (0,0);
\draw[thick,color=\red,postaction={mid arrow}] (-1,1) -- (0,0);

\node at (0,2.2) {\(\implies\)};

\draw[thick,color=\red,postaction={mid arrow}] (0,5) -- (0,4);
\draw[thick,color=\red,postaction={mid arrow}] (-1,5) -- (0,4);
\draw[thick,color=\blue,postaction={mid arrow}] (-1,5) -- (0,5);
\draw[thick,color=\paleblue,postaction={mid arrow}] (0,5) -- (1,5);
\node[cross out, line width=0.9pt, draw=\grey, minimum size=0.32cm, at={(0.45,5)}] {};

\foreach \x/\y in {0/0, 0/1, -1/0, -1/1, 1/0, 1/1, 0/4, 0/5, -1/4, -1/5, 1/4, 1/5} {\drawnode{(\x,\y)}
}

\node[at={(0,0)},label={[shift={(0.5,-0.5)}]\(\scriptstyle(i,j)\)}] {};
\node[at={(0,4)},label={[shift={(0.5,-0.5)}]\(\scriptstyle(i,j)\)}] {};
\end{tikzpicture}
\end{center}
Then \((i,j+1)\) lies in both the blue connector and its complementary red connector, but there is no vertical blue step out of it, contradicting \autoref{lemma:complementary_lemma}.
\end{proof}

\begin{proposition}
The map from the set of non-intersecting blue connectors to the set of non-intersecting red connectors defined by taking the complementary red connector is a weight-preserving bijection.
\end{proposition}

\begin{proof}
Write \(\Sigma C\) for the sum of the elements of a set \(C\).
Observe that \(\Sigma A^\comp - \Sigma B^\comp = \Sigma B - \Sigma A\),
and hence that the number of vertical steps made by a blue connector equals the number of diagonal steps made by a red connector.
Thus every vertical blue step in a non-intersecting blue connector must give rise to a diagonal red step in its complementary red connector.
That is, the nodes from which a non-intersecting blue connector takes vertical steps are precisely the nodes from which its complementary red connector takes diagonal steps (and these are precisely the intersections of the connectors).

It is then clear that a non-intersecting blue connector has the same weight as its complementary red connector, and that taking the analogous construction of a complementary blue connector provides an inverse.
\end{proof}

\subsection{Conclusion}

Combining the enumerations given by Propositions \ref{prop:l-path_count} and \ref{prop:r-path_count} with the bijection described in \autoref{section:path_bijection}, we obtain our main theorem, stated in full below.

\begin{theorem}\label{thm:main_theorem}
Let \(n\) be a nonnegative integer, let \(A\) and \(B\) be subsets of \(\intervalwz{n}\) of equal size% of size \(l\)
, and let \(A^\comp\) and \(B^\comp\) be their complements in \(\intervalwz{n}\).
Let \(\alpha\) and \(\beta\) be partitions with \(n\) parts (with parts equal to \(0\) permitted) such that \(\alpha_i \leq \beta_i\) for all \(i \in \intervalwz{n}\).
\parallelogramcondition
Then the determinants
\[
    \det\Bigl(%
        h_{b - a} (x_{\alpha_{a+1}+1}, x_{\alpha_{a+1}+2}, \ldots, x_{\beta_{b}}) %
    \Bigr)_{a \in A, b \in B}
\]
and
\[
    \det\Bigl(%
        e_{a' - b'} (x_{\alpha_{a'}+1}, x_{\alpha_{a'}+2}, \ldots, x_{\beta_{b'+1}}) %
    \Bigr)_{a' \in A^\comp, b' \in B^\comp}
\]
are equal.
\end{theorem}

\begin{remark}
The hypothesis in \autoref{thm:main_theorem} is precisely the hypothesis in \autoref{prop:r-path_count} which ensures that each entry of the second matrix counts the number of red paths correctly (by requiring that all minimal parallelograms between appropriate sinks and sources lie inside the Young diagram).
This hypothesis is necessary and sufficient for each individual entry to give a correct count.
However, this hypothesis is not necessary for the determinant to correctly count the total number of non-intersecting red connectors.
For example, suppose there exists \(m \in \intervalwz{n}\) such that \(\abs{A \cap \intervalwz{m}} = \abs{B \cap \intervalwz{m}}\).
Then no non-intersecting connector crosses the \((m+1)\)th row, and so the count of non-intersecting connectors is the product of the counts of non-intersecting connectors on each half of the lattice.
Meanwhile the matrices whose entries count paths are block triangular, and the entries of the off-diagonal block are irrelevant to the determinant, and so it is not necessary for the hypothesis to hold for pairs with indices on both sides of \(m+1\).
\end{remark}

Partitions whose parts are at most \(1\) less than the preceding part clearly satisfy the hypothesis of \autoref{thm:main_theorem}, and so we recover \autoref{thm:intro_det_identity} given in the introduction.

\subsection{Flagged Schur polynomial identity}
\label{subsection:flagged_polys}
We here state our main result in terms of flagged skew Schur polynomials.
Flagged Schur polynomials were introduced in the study of Schubert polynomials by Lascoux and Sch\"{u}tzenberger in \cite[Annexe~II]{lascoux1982Schubert} as determinants in symmetric polynomials.
These polynomials are sometimes also defined as generating functions for semistandard Young tableaux.
These definitions, however, are not always equivalent; both definitions, as well as sufficient conditions for their equivalence, are given in \cite[\S3]{wachs1985flaggedSchur}.

We first recall the determinantal definition.

\begin{definition}
Let \(\la\) and \(\mu\) be partitions with \(l\) parts such that \(\mu_i \leq \la_i\).
Let \(f\) and \(g\) be sequences of positive integers. % of length \(l\).
The determinantal row-flagged skew Schur polynomial of shape \(\la/\mu\) with flags \(f\) and \(g\) is
\[
    s_{\la/\mu}(f,g) = \det\Bigl(%
        h_{\la_i - \mu_j -i +j}(x_{f_j}, \ldots, x_{g_i})
    \Bigr)_{1 \leq i,j \leq l},
\]
and the determinantal column-flagged skew Schur polynomial of shape \(\la/\mu\) with flags \(f\) and \(g\) is
\[
    s^\ast_{\la/\mu}(f,g) = \det\Bigl(%
        e_{\la_i - \mu_j -i +j}(x_{f_j}, \ldots, x_{g_i})
    \Bigr)_{1 \leq i,j \leq l}.
\]
\end{definition}

We can rewrite our identity \autoref{thm:main_theorem} as follows, yielding a duality theorem for determinantal flagged skew Schur functions.
Recall that we index sets from smallest elements to largest (so, for example, \(A_1 < \ldots < A_l\)).
The \emph{conjugate} of a partition \(\la\) is denoted \(\la'\), obtained by reflecting the Young diagram in the main diagonal; conjugation is related to set complementation by \cite[(1.7), p.~3]{macdonald1998symmetric}.

\begin{theorem}\label{thm:flagged_polynomial_identity}
Let \(n\) be a nonnegative integer, let \(A\) and \(B\) be subsets of \(\intervalwz{n}\) of equal size \(l\)% of size \(l\)
, and let \(A^\comp\) and \(B^\comp\) be their complements in \(\intervalwz{n}\).
Let \(\alpha\) and \(\beta\) be partitions with \(n\) parts (with parts equal to \(0\) permitted) such that \(\alpha_i \leq \beta_i\) for all \(i \in [n]\).
\parallelogramcondition
Set:
%% this assumes A_1 is smallest element of A
\begin{align*}
    \mu_i   &= A_{l+1-i} + i - l, &\, 
    \la_i       &= B_{l+1-i} + i - l, \\
    f_i         &= \alpha_{A_{l+1-i}+1}+1, &\,
    g_i         &= \beta_{B_{l+1-i}}, \\
    f^\ast_i    &= \alpha_{A^\comp_{i}}+1, &\,
    g^\ast_i    &= \beta_{B^\comp_{i}+1}.
\end{align*}
Then
\[
    s_{\la/\mu}(f,g) = s^\ast_{\la'/\mu'}(f^\ast, g^\ast).
\]
\end{theorem}

\newcommand{\stabl}{\bar{s}}
We now consider the definition of flagged Schur polynomials as generating functions for tableaux.
We recall the definition from \cite[\S3]{wachs1985flaggedSchur}, denoting the polynomials thus defined as \(\stabl\) to distinguish them clearly from the determinantal \(s\).

\newcommand{\T}{\mathcal{T}}
\begin{definition}
Let \(\la\) and \(\mu\) be partitions such that \(\mu_i \leq \la_i\).
Let \(f\) and \(g\) be sequences of positive integers.
Let \(\T(\la/\mu, f, g)\) be the set of semistandard Young tableaux of shape \(\la/\mu\) such that for each \(i\) the entries in the \(i\)th row are bounded below by \(f_i\) and above by \(g_i\),
and let \(\T^\ast(\la/\mu, f, g)\) be the set of semistandard Young tableaux of shape \(\la/\mu\) such that for each \(i\) the entries in the \(i\)th column are bounded below by \(f_i\) and above by \(g_i\).
Given a tableaux \(t\), let \(x^t\) denote the monomial whose power of \(x_i\) is the number of occurences of \(i\) in \(t\).
The tableaux-generated row-flagged skew Schur polynomial of shape \(\la/\mu\) with flags \(f\) and \(g\) is
\[
    \stabl_{\la/\mu}(f,g) = \sum_{t \in \T(\la/\mu,f,g)} x^t,
\]
and the tableaux-generated column-flagged skew Schur polynomial of shape \(\la/\mu\) with flags \(f\) and \(g\) is
\[
    \stabl^\ast_{\la/\mu}(f,g) = \sum_{t \in \T^\ast(\la'/\mu',f,g)} x^t.
\]
\end{definition}

In \cite[Theorems 3.5 and 3.5*]{wachs1985flaggedSchur}, sufficient conditions are given for when the determinantal and tableaux-generated flagged Schur polynomials are equal.
Using these conditions, it can be shown that under the hypotheses of \autoref{thm:flagged_polynomial_identity} the determinantal and tableaux-generated flagged Schur polynomials are equal.
We thus obtain a duality theorem for tableaux-generated flagged Schur polynomials by replacing ``\(s\)'' with ``\(\stabl\)'' in \autoref{thm:flagged_polynomial_identity}.

The conditions from \cite[Theorems 3.5 and 3.5*]{wachs1985flaggedSchur} are not to be overlooked: they can fail frequently.
Consequently, tableaux-generated duality can hold even when the determinantal duality does not.
This is illustrated in the following example.

\ytableausetup{smalltableaux}
\begin{example}
Suppose \(l = 1\), \(r=2\), \(A = \set{0}\), \(B=\set{2}\), \(\alpha = (0, 0)\) and \(\beta = (m,1)\) for some \(m \geq 1\).
Then \(\mu = (0)\), \(\la=(2)\), \(f=(1)\), \(g=(1)\), \(f^\ast=(1,1)\) and \(g^\ast = (m,1)\).
Observe that the sets \(\T(\la/\mu, f,g)\) and \(\T^\ast(\la/\mu,f^\ast, g^\ast)\) are equal (containing only the tableau \raisebox{-1pt}{\ytableaushort{11}}) regardless of the value of \(m\), and so we have
\[
     \stabl_{\la/\mu}(f,g) = x_1^2 = \stabl^\ast_{\la'/\mu'}(f^\ast, g^\ast).
\]

However, the hypotheses of \autoref{thm:flagged_polynomial_identity} are met if and only if \(m \in \set{1,2}\).
We have, for all \(m\), that
\[
s_{\la/\mu}(f, g)
    = \det \begin{psmallmatrix}
h_2(x_1)
\end{psmallmatrix}
    = x_1^2,
\]
but when \(m \geq 3\) we have
\[
s^\ast_{\la'/\mu'}(f^\ast, g^\ast)
    = \det \begin{psmallmatrix}
e_1(x_1, \ldots, x_m) & e_2(x_1, \ldots, x_m) \\ 1 & x_1
\end{psmallmatrix}
%    &= x_1 e_1(x_1, \ldots, x_m) - e_2(x_1, \ldots, x_m) \\
    = x_1^2 - e_2(x_2, \ldots, x_m).
    %\neq s_{\la/\mu}(f,g).
    % &\neq x_1^2 \\ &= s_{\la/\mu}(f,g).
\]
Thus determinantal duality fails despite the tableaux-generated duality holding.
\end{example}

\section{Specialisations of the main theorem}\label{section:corollaries}

In this section we indicate how to recover Gessel and Viennot's binomial duality theorem (\autoref{thm:binomial_identity}) and Aitken's symmetric function duality theorem (\autoref{thm:sym_function_identity}) from our main thereom.
We also deduce lifts of Gessel and Viennot's theorem to \(q\)-binomial coefficients (\autoref{cor:q-binomial_identity}) and to symmetric polynomials (\autoref{cor:sym_poly_binomial_identity}).

To deduce \autoref{thm:binomial_identity}, \autoref{cor:q-binomial_identity} and \autoref{cor:sym_poly_binomial_identity}, we use staircase-shaped lattices.

\sympolybinomidentity*

\begin{proof}
In \autoref{thm:main_theorem}, take \(\beta = (n^{n})\) and take \(\alpha\) to be the staircase given by \(\alpha_i = n-i\) for \(i \in [n]\).
Then the left-hand matrix is \((h_{b-a}(x_{n-a}, \ldots, x_{n}))_{a \in A,b \in B}\) and the right-hand matrix is \((e_{a'-b'}(x_{n-a'+1}, \ldots, x_{n}))_{a' \in A^\comp,b \in B^\comp}\).
Relabelling the variables (via \(x_i \mapsto x_{n+1-i}\)) gives the result.
\end{proof}

The \(q\)-binomial coefficients (also known as Gaussian coefficients) are polynomials in \(q\), and are a generalisation of the usual binomial coefficients in the sense that setting \(q=1\) yields the corresponding binomial coefficients.
They are defined (see for example \cite[\S3]{konvalina2000binomialcoefficients}) by 
\[
    \qbinom{n}{k} = \frac{(1-q^n)(1-q^{n-1})\cdots(1-q^{n-k+1})}{(1-q)(1-q^{2})\cdots(1-q^{k})}.
\]

\qbinomialidentity*

\begin{proof}
Recall that the \(q\)-binomial coefficients are related to the symmetric polynomials in the following way \cite[\S3, equations 17 and 18]{konvalina2000binomialcoefficients}:
\begin{align*}
    h_{k}(1, q, \ldots, q^{n-1}) &= \qbinom{n+k-1}{k}, \\
    e_{k}(1, q, \ldots, q^{n-1}) &= q^{\binom{k}{2}} \qbinom{n}{k}.
\end{align*}
Thus, set \(x_i = q^{i-1}\) in \autoref{cor:sym_poly_binomial_identity} and the matrix entries become, respectively,
\[
h_{b-a}(1,q, \ldots, q^{a})
    = \qbinom{b}{b-a} = \qbinom{b}{a}
\]
and
\[
e_{a'-b'}(1,q, \ldots, q^{a'-1})
    = q^{\binom{a'-b'}{2}} \qbinom{a'}{a'-b'} = q^{\binom{a'-b'}{2}}\qbinom{a'}{b'},
\]
as required.
\end{proof}

Setting \(q=1\) in \autoref{cor:q-binomial_identity} recovers Gessel and Viennot's binomial duality theorem (\autoref{thm:binomial_identity}).

To recover Aitken's symmetric function duality theorem (\autoref{thm:sym_function_identity}), we use rectangular lattices.
In \autoref{thm:main_theorem}, take \(\alpha = (0^n)\) and \(\beta = (m^{n})\), for a positive integer \(m\), to obtain
\[
    \det\Bigl(%
        h_{b - a} (x_{1}, x_{2}, \ldots, x_{m}) %
    \Bigr)_{a \in A, b \in B}
    =
    \det\Bigl(%
        e_{a' - b'} (x_{1}, x_{2}, \ldots, x_{m}) %
    \Bigr)_{a' \in A^\comp, b' \in B^\comp}
    .
\]
Since this holds for arbitrary positive integers \(m\),
\autoref{thm:sym_function_identity} follows.

\section{Insufficiency of Jacobi's complementary minor formula}
\label{section:JT_connection}

\autoref{thm:sym_function_identity} was proved in \cite{aitken1931duality} using Jacobi's complementary minor formula and a form of Newton's identity.
These two results are stated below.
We here outline Aitken's proof, and show that this method is not sufficient to deduce our main theorem.

For our purposes it is convenient to index matrix rows and columns from \(0\).
Given a \((d+1) \times (d+1)\) matrix \(M\) and subsets \(A,B \subseteq \intervalwz{d}\), let \(M_{A,B}\) denote the matrix obtained by retaining only those rows indexed by elements of \(A\) and those columns indexed by elements of \(B\).
Write \(\Sigma A\) for the sum of the elements of \(A\).

\begin{proposition}[Jacobi's complementary minor formula { \cite[Lemma A.1(e), p.~96]{caracciolo2013identities}}]
\label{prop:Jacobi_comp_minor}
Let \(M\) be an invertible \((d+1) \times (d+1)\) matrix and let \(A,B \subseteq \intervalwz{d}\) be subsets of equal size.
Then
\[
    \det ( M_{A,B} ) = (-1)^{\Sigma A + \Sigma B} \det (M) \det \mleft( {\mleft((M^{-1})^{\top}\mright)}_{A^\comp, B^\comp} \mright).
\]
\end{proposition}

\begin{proposition}[Newton's identity {\cite[Equation (\ensuremath{2.6'}), p.~21]{macdonald1998symmetric}}]
\label{prop:Newton_identity}
Let \(d > 0\).
Then
\[
    \sum_{i=0}^d {(-1)}^i e_i h_{d-i} = 0.
\]
\end{proposition}

In \cite{aitken1931duality}, Aitken obtains his identity (\autoref{thm:sym_function_identity}) by applying \autoref{prop:Jacobi_comp_minor} to the matrix \({( h_{j-i} )}_{0\leq i,j \leq n}\).
Its inverse is \({( {(-1)}^{i+j} e_{j-i})}_{0\leq i,j \leq n}\), as can be verified using \autoref{prop:Newton_identity}. 

If we attempt to use this method to prove \autoref{thm:main_theorem}, we would be required to show, given partitions \(\alpha\) and \(\beta\) satisfying the hypotheses, that the matrices
\begin{align*}
H(\beta / \alpha) &= {\Bigl(
    h_{j-i}(x_{\alpha_{i+1}+1}, \ldots, x_{\beta_j})
\Bigr)}_{0\leq i,j \leq n}\\
\intertext{and}
E(\beta / \alpha) &= {\mleft(
    {(-1)}^{i+j} e_{j-i}(x_{\alpha_j+1}, \ldots, x_{\beta_{i+1}})
\mright)}_{0\leq i,j \leq n}
\end{align*}
are inverse.
However, \autoref{thm:main_theorem} can provide a determinant identity when this is not the case.

For example, let \(n=3\) and let \(\alpha = (2, 0, 0)\) and \(\beta = (3,3,1)\).
We have
\begin{align*}
    {\mleft(E(\beta / \alpha)H(\beta / \alpha)\mright)}_{0,2}
    &= h_2(x_3) - e_1(x_3)h_1(x_1, x_2, x_3) + e_2(x_1, x_2, x_3) \\
    &= x_1 x_2 \\
    &\neq 0,
\end{align*}
so the matrices \(H(\beta / \alpha)\) and \(E(\beta / \alpha)\) are not inverse.
Nevertheless, choosing \(A = \set{0,1,2}\) and \(B = \set{1,2,3}\) meets the hypotheses of \autoref{thm:main_theorem}, and so we find that the determinants
\[
\hspace*{-6.5em}
\begin{blockarray}{cc>{\scriptstyle}c>{\scriptstyle}c>{\scriptstyle}c}
      &   & b=1 & b=2 & b=3 \\
      &   & \beta_{b}=3 & \beta_{b}=3 & \beta_{b}=1\\[4pt]
    \begin{block}{>{\scriptstyle}c>{\scriptstyle}c|ccc|}
    a=0 & \alpha_{a+1}+1=3 & h_1(x_3) & h_2(x_3) & 0 \\
    a=1 & \alpha_{a+1}+1=1 & 1 & h_1(x_1, x_2, x_3) & h_2(x_1) \\
    a=2 & \alpha_{a+1}+1=1 & 0 & 1 & h_1(x_1) \\
    \end{block}
\end{blockarray}
\]
and
\[
\hspace*{-6.5em}
\begin{blockarray}{ccc>{\scriptstyle}c}
      &   && b'=0 &\\
      &   && \beta_{b'+1}=3 & \\[4pt]
    \begin{block}{>{\scriptstyle}c>{\scriptstyle}cc|c|}
    a'=3 & \alpha_{a'}+1=1 && e_3(x_1, x_2, x_3)\\
    \end{block}
\end{blockarray}
\]
are equal.

\bibliographystyle{alpha}
\bibliography{references}

\begin{thebibliography}{Wac85}

\bibitem[Ait31]{aitken1931duality}
{\initials{A.C.}}~Aitken.
\newblock Note on dual symmetric functions.
\newblock {\em Proceedings of the Edinburgh Mathematical Society}, 2(3):164--167, 1931.

\bibitem[BC05]{benjamin2005countingOnDeterminants}
Arthur~T. Benjamin and Naiomi~T. Cameron.
\newblock Counting on determinants.
\newblock {\em The American Mathematical Monthly}, 112(6):481--492, 2005.

\bibitem[CSS13]{caracciolo2013identities}
Sergio Caracciolo, Alan~D. Sokal, and Andrea Sportiello.
\newblock Algebraic/combinatorial proofs of {C}ayley-type identities for derivatives of determinants and pfaffians.
\newblock {\em Advances in Applied Mathematics}, 50(4):474--594, Apr 2013.

\bibitem[GV85]{gessel1985paths}
Ira Gessel and G{\'e}rard Viennot.
\newblock Binomial determinants, paths, and hook length formulae.
\newblock {\em Advances in Mathematics}, 58(3):300--321, 1985.

\bibitem[KM59]{karlin1959coincidence}
Samuel Karlin and James McGregor.
\newblock Coincidence probabilities.
\newblock {\em Pacific Journal of Mathematics}, 9(4):1141--1164, 1959.

\bibitem[Kon00]{konvalina2000binomialcoefficients}
John Konvalina.
\newblock A unified interpretation of the binomial coefficients, the {Stirling} numbers, and the {Gaussian} coefficients.
\newblock {\em The American Mathematical Monthly}, 107(10):901--910, 2000.

\bibitem[Lin73]{lindstroem1973matroids}
Bernt Lindstr\"{o}m.
\newblock On the vector representations of induced matroids.
\newblock {\em Bulletin of the London Mathematical Society}, 5(1):85--90, 1973.

\bibitem[LS82]{lascoux1982Schubert}
Alain Lascoux and Marcel-Paul Sch\"{u}tzenberger.
\newblock Polyn\^{o}mes de {S}chubert.
\newblock {\em Comptes rendus des s\'{e}ances de l'Acad\'{e}mie des Sciences, S\'{e}rie I Math\'{e}matique}, 294(13):447--450, 1982.

\bibitem[Mac98]{macdonald1998symmetric}
{\initials{I.G.}}~Macdonald.
\newblock {\em Symmetric Functions and Hall Polynomials}.
\newblock Oxford classic texts in the physical sciences. Clarendon Press, 1998.

\bibitem[Ste90]{stembridge1990nonintersecting}
John~R. Stembridge.
\newblock Nonintersecting paths, pfaffians and plane partitions.
\newblock {\em Advances in Mathematics}, 83:96--131, 1990.

\bibitem[Vie85]{viennot1985orthogonal}
G{\'e}rard Viennot.
\newblock A combinatorial theory for general orthogonal polynomials with extensions and applications.
\newblock In {\em Polyn{\^o}mes Orthogonaux et Applications}, volume 1171 of {\em Lecture Notes in Mathematics}, pages 139--157. Springer Berlin Heidelberg, 1985.

\bibitem[Wac85]{wachs1985flaggedSchur}
Michelle~L. Wachs.
\newblock Flagged {S}chur functions, {S}chubert polynomials, and symmetrizing operators.
\newblock {\em Journal of Combinatorial Theory, Series A}, 40(2):276--289, 1985.

\end{thebibliography}

\end{document}